\documentclass[12pt]{article}
\usepackage{e-jc}

\usepackage{amsthm,amsmath,amssymb}
\usepackage{graphicx}
\usepackage{ytableau}
\usepackage{comment}
\usepackage[numbers]{natbib}
\usepackage{tikz}
\usetikzlibrary{arrows,chains,matrix,positioning,scopes}

\newcommand{\card}[1]{{\left|{#1}\right|}}
\newcommand{\bra}[1]{{\left({#1}\right)}}
\newcommand{\floor}[1]{{\lfloor{#1}\rfloor}}

\newcommand{\set}[1]{{\left\{{#1}\right\}}}
\newcommand{\Z}{\mathbb{Z}}
\newcommand{\N}{\mathbb{N}}

\newcommand{\F}{\mathbb{F}}
\newcommand{\qbinom}[2]{\left[{#1 \atop #2}\right]_q}
\newcommand{\eqdef}{:=}
\newcommand{\lcm}{\textup{lcm}}

\dateline{April 24, 2018}{May 29, 2020}{Jun 26, 2020}

\MSC{05A10, 05A15}

\Copyright{The author. Released under the CC BY license (International 4.0).}

\title{Coefficients of Gaussian Polynomials modulo $N$}

\author{Dylan Pentland\thanks{Supported by NSF grant no. DMS-1519580.}\\
\small Department of Mathematics\\[-0.8ex]
\small Massachusetts Institute of Technology\\[-0.8ex] 
\small Massachusetts, U.S.A.\\
\small\tt dylanp@mit.edu}

\begin{document}

\maketitle

\begin{abstract}
Let $\qbinom{n}{k}$ be a $q$-binomial coefficient. Stanley conjectured that the function $f_{k}(n) = \#\left\{\alpha : [q^{\alpha}] \left[{n \atop k}\right]_q  \equiv R \pmod{N}\right\}$ is quasipolynomial for $N$ prime. We prove this for any integer $N$ and obtain an expression for the generating function $F_{k}(x)$ for $f_{k}(n)$.
\end{abstract}

\section{Introduction}
The $q$-analogue of the binomial coefficient is typically denoted $\qbinom{n}{k}$ and is defined by the rational expression 
\[\qbinom{n}{k} = \frac{[n]!}{[n-k]![k]!},\]
where $[n]! = \prod_{i=1}^{n}{(1-q^i)/(1-q)}$. These are polynomials with degree $k(n-k)$.

These polynomials appear in combinatorics and have connections to the theory of symmetric polynomials as well as representation theory. In particular, an important characterization is that they enumerate the Grassmannian $\mathbf{Gr}(k , \F_q^n)$:

\begin{theorem}\label{11}
The number of $k$-dimensional subspaces of $\F_q^n$ is $\qbinom{n}{k}$.
\end{theorem}

\noindent For a proof see, for example, \cite{stanley}. While $q$-binomial coefficients are common objects in combinatorics, recent works such as \cite{kronecker} or \cite{stanley2} have sparked additional interest in these objects and their coefficients.

In this paper, we investigate the behavior of these coefficients modulo some positive integer $N\in \N$. One motivation for this is the classical Lucas' theorem:
\begin{theorem}[Lucas' Theorem] For $p$ prime, let $n,k\in \N$ have base $p$ expansions $n = \sum_{i\ge 0}{n_i p^i}, k =\sum_{i\ge 0}{k_i p^i}$. Then
\[\binom{n}{k} \equiv \prod_{i\ge 0}{\binom{n_i}{k_i}} \pmod{p}.\]
\label{12}
\end{theorem}

\noindent By fixing $k$, the values of $\binom{n}{k} \pmod{p}$ can be shown to form a repeating sequence related to the base $p$ expansion of $k$. This extends to modulo $N$, as seen by the following corollary from \cite{kwongbinom}:

\begin{theorem}[Kwong \cite{kwongbinom}]
Let the prime factorization of $N$ be given by $\prod{p_i^{e_i}}$ for primes $p_i$. Then $\binom{n}{k}$ is purely periodic modulo $N$ for fixed $k$, with period
\[P = \prod{p_i^{e_i+b_i-1}},\]
where $b_i\in \N$, $p_i^{b_i-1}<k<p_i^{b_i}$.
\label{13}
\end{theorem}

\noindent Here, the term purely periodic means that a sequence $(x_n)_{n \in \N}$ has $x_{n}=x_{n+Q}$ for some $Q$ and all $n\in \N$. The $q$-binomial coefficients are an example of a ``$q$-analogue", in the sense that $\lim_{q\to 1}\qbinom{n}{k}=\binom{n}{k}$. As a result, it is reasonable to expect similar structured behavior modulo $p$ or even with general composites in the coefficients of $\qbinom{n}{k}$, since this shows $\binom{n}{k} = \lim_{q\to 1}{\qbinom{n}{k}}=\sum_{i\ge 0}{[q^i]\qbinom{n}{k}}$. Here, $[q^i]f(q)$ denotes the coefficient of $q^i$ in $f$.

We prove and generalize Conjecture \ref{16}, that the ``residue counting" function for these coefficients is a \textit{quasipolynomial}. From \cite{stanley}, we have the following definition of a quasipolynomial function:

\begin{definition}
A function $f: \N \to \mathbb{C}$ is \textit{quasipolynomial} with degree $d$ if
\[f(n) = c_d(n)n^d + c_{d-1}(n)n^{d-1} + \ldots + c_0(n)\]
where each $c_i(n)$ is a periodic function with integer period and $c_d(n)$ is not identically $0$. We call $Q$ a \textit{quasiperiod} of $f$ if it is a common period of all $c_i(n)$. Note that $Q$ is not unique, since $kQ$ is a quasiperiod for $k\in \N$.
\label{14}
\end{definition}

\noindent Equivalently, we can say $f(n)=P_i(n)$ for $n \equiv i \pmod{Q}$ where $P_i \in \Z[x]$. In order to state the main result (Theorem \ref{19}), we make the following definitions.

\begin{definition}
For a natural number $N$ that we call the \textit{modulus}, $R \in \Z/N \Z$, and $k \in \N$, we define
\[f_{k}(n) = \#\set{\alpha : [q^{\alpha}] \qbinom{n}{k}  \equiv R \pmod{N}}.\]
This function counts the number of coefficients congruent to $R$ modulo $N$.
\label{15}
\end{definition}

\begin{remark}
From \cite{kparts}, we see that $p_{\le k}(n)$, the number of partitions of $n$ with at most $k$ parts, is also an example of a quasipolynomial function.
\end{remark}

\begin{definition}
Define $\pi_N(k)$ as the minimal period of $p_{\le k}(n)$ modulo $N$.
\label{17}
\end{definition}

\begin{definition}

Define $\pi'_N(k)$ as follows:
\[\pi'_N(k) = \pi'_N(k-1)\lcm \left(N,\frac{\pi_N(k)}{\pi_N(k-1)}\right). \]
We set $\pi'_N(1)=1$.
\label{18}
\end{definition}

\noindent This definition makes it so that $N\mid \frac{\pi'_N(k)}{\pi'_N(k-1)}$ for $k>1$. Stanley originally conjectured the following in \cite{stanley3}:

\begin{conjecture}[Stanley \cite{stanley3}]
The function $f_{k}$ is quasipolynomial for $N$ prime.
\label{16}
\end{conjecture}

\noindent The following theorem, which generalizes Conjecture \ref{16}, is the main result of this paper. This is shown in Sections 3 and 4.

\begin{theorem}
For a modulus $N$, the function $f_{k}(n)$ is quasipolynomial, with a quasiperiod $\pi'_{N}(k)$ and degree one.
\label{19}
\end{theorem}

\noindent The idea will be to formulate an equivalent restatement in Theorem \ref{main}, which makes a more direct statement about the structure of the coefficients modulo $N$. In Section 5, we investigate the structure of the generating function
\[F_{k}(x) = \sum_{n\ge k}{f_{k}(n)x^n}.\]
In section 6, we investigate some asymptotics of the proven quasiperiod and conjectured minimal quasiperiod.

\section[Coefficients of low degree terms]{Coefficients of low degree terms in $\qbinom{n}{k}$}

We first try to understand the behavior of the coefficient of $q^i$ in $\qbinom{n}{k}$ for small $i$. The following result is well-known and follows from the identity $\sum_{i\ge 0}{p(n,k,i)q^i} = \qbinom{n+k}{k}$, where $p(n,k,i)$ denotes the number of partitions $\lambda \vdash i$ with at most $k$ parts and maximal part $\le n$. 

\begin{lemma}
 \label{21}
Let $n_0 , k\in \N$ be arbitrary, and $n \geq n_0 + k$. Then for $0 \leq i < n_0$, we have $[q^i]\qbinom{n}{k} = p_{\leq k}(i)$.
\end{lemma}

\begin{remark}
A similar result is true for the last $n_0$ coefficients by the symmetry of the $q$-binomial coefficients.
\end{remark}
This warrants an investigation of the function $p_{\le k}(i)$ modulo $N$. The following theorem from \cite{NijWilf} shows that it is purely periodic, and determines the minimal period for primes. Then in \cite{kwongpart}, this is extended to prime powers. By the Chinese Remainder Theorem, understanding the behavior of $p_{\le k}(i)$ modulo prime powers is sufficient to understand its behavior modulo $N$.

\begin{theorem}[Kwong]
For a prime power $p^e$, fix a set $S=\set{s_0, s_1, \ldots s_l}$ with entries in $\mathbb{N}$. Let $p(n;S)$ be given by the generating function
\[P(x;S) \eqdef \prod_{s\in S}{\frac{1}{1-x^s}} = \sum_{n\ge 0}{p(n;S)x^n},\]
hence $p(n;S)$ is the number of partitions $\lambda$ with parts in $S$ and $|\lambda |=n$. Then $p(*;S)$ is purely periodic modulo $p^e$, with minimal period 
\[\pi_{p^e}(S) = p^{b_p(S)+e-1} L_p(S)\] 
where $b_p(S)$ is the smallest integer such that 
\[p^{b_p(S)} \ge \sum_{s\in S}{p^{\nu_p(s)}}\]
where $\nu_p(s)$ is the $p$-adic valuation of $s$ and $L_p(S)\eqdef \textup{lcm}(S)/p^{\nu_p(\textup{lcm}(S))}$ is the $p$-free part of $\textup{lcm}(S)$.
\label{22}
\end{theorem}
For a more detailed discussion, see \cite{kwongpart}. Lemma \ref{21} then shows that for $n_0$ sufficiently large, the first $n_0$ coefficients of $\qbinom{n}{k}$ for $n-k\ge n_0$ will follow a repeating pattern of period $\pi_{p^e}(k)=p^{b_p([k])+e-1} L_p([k])$ modulo $p^e$. Here, $[k] \eqdef \set{1,2,\ldots k}$.

It is worth noting that the statement is slightly incorrect: the theorem does not hold in the trivial case $k=1$ where $\pi_N(k)=1$. However, there seem to be no other errors otherwise with the proof. Fortunately, this case is simple and can be ignored. From now on, we have $k\ge 2$.

\begin{definition}
Fix $k$, a modulus $N \in \N$, and $n > \pi_{N}(k)$. Let $\mathcal{S}$ be the sequence of the first $\pi_N(k)$ coefficients of $\qbinom{n+k}{k}$ reduced mod $N$. It is given by $\mathcal{S} = (s_0, s_1, \ldots s_{\pi_{N}-1})$, where
$s_\alpha \equiv  [q^\alpha]\qbinom{n+k}{k} \pmod{N}$.

\label{buf}
\end{definition}

Theorem \ref{22} and Lemma \ref{21} show that $\mathcal{S}$ determines the periodic sequence $p_{\le k}(i)$ modulo $N$.

\begin{example}
One example of this sequence for $N=2$ and $k=3$ is shown in Figure \ref{table}.
\begin{figure}[ht]
\begin{center}

 \begin{tabular}{|c c c c c c c c c c c c|} 
 \hline
 $s_0$ & $s_1$ & $s_2$ & $s_3$ & $s_4$ & $s_5$ & $s_6$ & $s_7$ & $s_8$ & $s_9$ & $s_{10}$ & $s_{11}$  \\ [0.5ex] 
 \hline
 1 & 1 & 0 & 1 & 0 & 1 & 1 & 0 & 0 & 0 & 0 & 0 \\ 
 \hline
\end{tabular}
\caption{Values of $\mathcal{S}$ modulo $2$.}
\label{table}
\end{center}
\end{figure}

\end{example}

Next, we study $\mathcal{S}$ for prime powers $p^e$ as the modulus. The generating function of $p_{\le k}(n)$ is given by
\[P_{\le k}(q) \eqdef \sum_{n\ge 0}{p_{\le k}(n)q^n} = \frac{1}{\prod_{i\in[k]}{1-q^i}}.\]

\begin{lemma} $\textup{lcm}([k]) \mid \pi_{p^e}(k) $.
\label{24}
\end{lemma}

\begin{proof}
By Theorem 2.2, it suffices to show that $b_p([k]) \ge \nu_p(\text{lcm}([k]))$. We can then see
\[p^{b_p([k])} \ge \sum_{i\in[k]}{p^{\nu_p(i)}} \ge p^{\max_{i \in [k]}{\nu_p(i)}} = p^{\nu_p(\text{lcm}([k])}.\]
Taking logs, the result follows.
\end{proof}

\begin{definition}
Define the operator $\Delta_{Q} : \Z[[q]] \rightarrow \Z[[q]]$ for $Q \in \N$ via the formula \begin{align*}
    (\Delta_{Q} F)(q)
    &= F(q) - q^{Q} F(q) \\
    &= \sum_{n\ge 0}{f(n)q^n} - \sum_{n\ge Q}{f(n-Q)q^n},
\end{align*}
This can be viewed as an analogue of the finite difference operator ($\Delta$, as in \cite{stanley} \S 1.9) acting on formal power series.
\label{25}
\end{definition}

\begin{example}
Consider the generating function $F(q) = \frac{q}{(1-q)^2} = \sum_{i\ge 0}{iq^i}$. Suppose we want to calculate $\Delta_5 F(q)$: then we have
\begin{align*}
\Delta_5 F(q) &= \frac{q}{(1-q)^2}-\frac{q^6}{(1-q)^2} \\
&= \frac{5q^5}{1-q}+(4q^4+3q^3+2q^2+1q).
\end{align*}
This demonstrates a key aspect of the $\Delta_Q$ operator: the lowest $Q$ monomials  remain unchanged, while the rest of the sequence can be viewed as a union of $Q$ subsequences with the traditional finite difference operator applied.
\end{example}
\begin{theorem}\label{26}
Fix $n,k$. Consider the sequence $\mathcal{S}$ where we take coefficients modulo $N=p^e$. Then $s_{\card{\mathcal{S}}-1}, \ldots s_{\card{\mathcal{S}}-\binom{k+1}{2}+1}$ are all $0$ modulo $N$, and $s_i = (-1)^{k+1} s_{\pi_N(k)-\binom{k+1}{2}-i}$ for $\pi_N(k)-\binom{k+1}{2}-i\ge 0$ and $i\ge 0$.
\end{theorem}
\begin{proof}
The main idea behind this result is to exploit the simple form of the generating function $P_{\le k}(q)$. We can re-write it as follows, letting $Q := \pi_{p^e}(k)$:
\[P_{\le k}(q) = \frac{1}{\prod_{i\in[k]}{1-q^i}} = \frac{\gamma(q)}{(1-q^{Q})^k}, \tag{1}\]
where we can obtain
\[\gamma(q) = \frac{(1-q^{Q})^k}{\prod_{i\in[k]}{1-q^i}} = \prod_{d \mid Q}{\Phi_d(q)^{f(d)}} \tag{2}\]
where $f(d) =k-\#\{i \in [k] : d \mid i\} \ge 0$ and $\Phi_d$ denotes the $d$th cyclotomic polynomial. To show (2), note that $f(d)$ accounts for every factor in the denominator, and is bounded below by $0$. This follows from Lemma \ref{24} and the fact that each cyclotomic factor of the denominator can appear at most $k$ times (at most once for each factor $(1-q^i)$). Thus we conclude that $\gamma(q) \in \Z[q]$ and that $\deg \gamma(q) = kQ - \binom{k+1}{2}$. Using $\Delta_Q$ as in Definition \ref{25}, we obtain
\[ \Delta_{Q}^k P_{\le k}(q) = \gamma(q).\]
Using the fact that $P_{\le k}(q)$ can be written as $P_{\le k}(q) \equiv \frac{\gamma_0(q)}{1-q^{Q}} \pmod{p^e}$ for some unique $\gamma_0(q)$ with $\deg \gamma_0 < Q$ by Theorem \ref{22}, we can see that $\Delta_{Q} P_{\le k}(q) \equiv \gamma_0(q) \pmod{p^e}$. It follows from this that 
\[\Delta_{Q}^{k} P_{\le k}(q) \equiv \sum_{i\ge 0}{(-1)^i\binom{k-1}{i} \gamma_0(q) q^{Q i}} \pmod{p^e},\]
using the formula for $\Delta^k f(n)$ from \cite{stanley} in \S 1.9. This is straightforward to verify using induction. Thus, for $r \in \Z / Q \Z$ we have
\[[q^{r+Q(k-1)}]\gamma(q) \equiv (-1)^{k-1}[q^r]\gamma(q) \pmod{p^e}.\]
Knowing that $\deg \gamma = kQ - \binom{k+1}{2}$, there must be $\binom{k+1}{2}-1$ zeroes at the end of $\mathcal{S}$. Furthermore, the polynomial $\gamma(q)$ can be shown to be symmetric using (2) and the symmetry of the cyclotomic polynomials (this is only true for $\Phi_d$ when $d>1$, but $d=1$ is not an issue as $f(1)=0$). Referring to Figure \ref{sym}, this shows the symmetry of $\mathcal{S}$ when the trailing zeroes are ignored: by the symmetry of $\gamma(q)$, the elements with label $i$ in Figure \ref{sym} are equal. These are also identical instances of $\mathcal{S}$ without the trailing zeroes up to sign, so $\set{s_{0}, \ldots s_{\card{\mathcal{S}}-\binom{k+1}{2}}}$ is symmetric or ``anti-symmetric" about its center. Precisely, this says that $s_i = (-1)^{k+1} s_{\pi_N(k)-\binom{k+1}{2}-i}$.

\begin{figure}[ht]
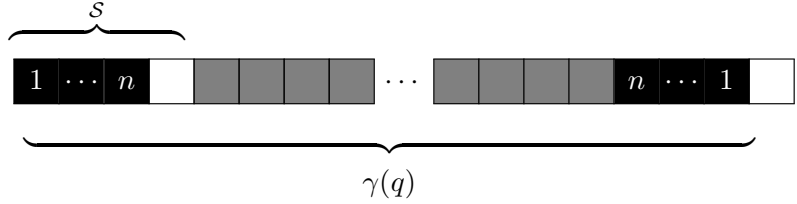


\ytableausetup{centertableaux}

\[\overbrace{\quad\quad\quad\quad\quad\quad}^{\mathcal{S}}\quad\quad\quad\quad\quad\quad\quad\quad\quad\quad\quad\quad\quad\quad\quad\quad\quad\quad\quad\quad\quad\]
\begin{center}
\begin{ytableau}
*(black) \color{white}{1} & *(black) \color{white}{\cdots} & *(black) \color{white}{\kappa} & & *(gray) & *(gray) & *(gray) & *(gray)\\
\end{ytableau}
$\cdots$
\begin{ytableau}
*(gray) & *(gray) & *(gray) & *(gray) & *(black) \color{white}{\kappa} & *(black) \color{white}{\cdots} & *(black) \color{white}{1} & \\
\end{ytableau}
\end{center}
\[\underbrace{\quad\quad\quad\quad\quad\quad\quad\quad\quad\quad\quad\quad\quad\quad\quad\quad\quad\quad\quad\quad\quad\quad\quad\quad\quad}\quad\]
\[\gamma(q)\quad\]
\caption{Symmetry in $\mathcal{S}$. White boxes represent sections of zeroes modulo $p^e$, gray sections represent the other values of coefficients of $\gamma(q)$ modulo $p^e$. The coefficients are ordered from left to right by increasing associated powers of $q$. We define $\kappa=Q-\binom{k+1}{2}$, so that the white numbers $1,2, \ldots,\kappa$ enumerate coefficients in the black sections.}
\label{sym}
\end{figure}

\end{proof}

These ideas can be generalized using the Chinese Remainder Theorem.

\begin{lemma}
\label{27}
Partitions with at most $k$ parts are purely periodic modulo $N$ for all $N \in \N$, with period
\[ \pi_N(k) = \underset{p|N}{\lcm}{\bra{\pi_{p^{\nu_p(N)}}(k)}}.\]
\end{lemma}

\begin{corollary}
Theorem \ref{26} also holds for $\mathcal{S}$ for arbitrary moduli $N$.
\label{28}
\end{corollary}

\begin{theorem}
Let $k\ge 0$ and $N$ be odd.  If $k$ is odd and $\textup{gcd}\bra{\frac{\pi_N(k+1)}{\pi_N(k)}, N}>1$, then we have
\[\frac{\pi_N(k+1)}{\pi_N(k)}\sum_{i\in \Z/\pi_N(k)\Z}{p_{\le k}(i)}\equiv 0 \pmod{N}.\]
Otherwise we have the stronger result $\sum_{i\in \Z/\pi_N(k)\Z}{p_{\le k}(i)}\equiv 0 \pmod{N}$.
\label{29}
\end{theorem}
\begin{proof}
First, we prove this for when $k$ is even. We have two cases. First, suppose $\frac{\pi_N(k)-\binom{k+1}{2}+1}{2}\not \in \Z$. This means there exists a ``central" element that is self-inverse ($x=-x$ mod $N$) in $\mathcal{S}$ by Corollary \ref{28}. Since $N$ is odd it is $0$ mod $N$. Using Corollary \ref{28} we pair all other terms in $\sum_{i\in \Z/\pi_N(k)\Z}{p_{\le k}(i)}$ in zero-sum pairs.

Otherwise, $\frac{\pi_N(k)-\binom{k+1}{2}+1}{2}\in \Z$. There is no central entry, and pairing via Corollary \ref{28} suffices to show $\sum_{i\in \Z/\pi_N(k)\Z}{p_{\le k}(i)} \equiv 0 \pmod{N}$.

For $k$ odd, we use a different method since $(-1)^{k+1}=1$. We have
\begin{align*}
\sum_{i\in \Z/\pi_N(k+1)\Z}{p_{\le k+1}(i)-p_{\le k}(i)} &\equiv \sum_{i\in \Z/\pi_N(k+1)\Z}{p_{=(k+1)}(i)} \pmod{N}\\
&\equiv \sum_{i\in \Z/\pi_N(k+1)\Z}{p_{\le k+1}(i-(k+1))} \pmod{N}\\
&\equiv 0 \pmod{N}. \quad (\text{by even case, shift invariance})
\end{align*}
Thus, we obtain
\[\frac{\pi_N(k+1)}{\pi_N(k)}\sum_{i\in \Z/\pi_N(k)\Z}{p_{\le k}(i)}\equiv 0 \pmod{N}\]
Unless $\text{gcd}\bra{\frac{\pi_N(k+1)}{\pi_N(k)}, N}>1$, the stronger statement  $\sum_{i\in \Z/\pi_N(k)\Z}{p_{\le k}(i)}\equiv 0 \pmod{N}$ holds since $\frac{\pi_N(k+1)}{\pi_N(k)}$ would be invertible modulo $N$.
\end{proof}

\begin{corollary}
Suppose $\sum_{i\in \Z/\pi_N(k)\Z}{p_{\le k}(i)}\equiv 0\pmod{N}$ for odd $N$. Then the same holds for modulo $2N$ and $\pi_{2N}$ given $\frac{\pi_2(k)-\binom{k+1}{2}+1}{2}\in \Z$. \label{210}
\end{corollary}

\begin{proof}
Because $\pi_N(k)\mid \pi_{2N}(k)$, we get that $\sum_{i\in \Z/\pi_{2N}(k)\Z}{p_{\le k}(i)}\equiv 0\pmod{N}$. Next, consider the sum modulo $2$. We have $\sum_{i\in \Z/\pi_2(k)\Z}{p_{\le k}(i)} \equiv 0 \pmod{2}$, since modulo $2$ the symmetry relation in Theorem \ref{26} becomes $s_i \equiv - s_{\pi_2(k)-\binom{k+1}{2}-i}$ because $1\equiv -1$ modulo $2$ - hence, the reasoning in Theorem \ref{29} even $k$ applies to all $k$. Since $\pi_2(k) \mid \pi_{2N}(k)$, we see that
\[\sum_{i\in \Z/\pi_{2N}(k)\Z}{p_{\le k}(i)}\equiv 0\pmod{2}.\]
Using the Chinese Remainder Theorem, this is $0$ modulo $2N$.
\end{proof}

\section[Decomposition into sections]{Decomposition of $\qbinom{n}{k}$}

In this section and the next, we exploit the results from Section 2 regarding the periodicity of $\mathcal{S}$ and the structure of $\mathcal{S}$ (as described by Theorem \ref{26}) in order to prove Theorem \ref{18}.

\begin{definition}
For a modulus $N$, we define the function $\mathcal{L}_{k}^{(i)}$ by
$$f_{k}(n) = \mathcal{L}_{k}^{(i)}\left( \frac{n-i}{\pi'_N(k)}\right),$$
if $n\equiv i \pmod{\pi'_N(k)}$.
\label{31}
\end{definition}

\begin{remark}
The change of variables $n \mapsto \frac{n-i}{\pi'_{N}(k)}$ is used to simplify proofs.
\end{remark}
The aim is now to show that the functions $\mathcal{L}_{k}^{(i)}$ are linear, from which it follows by definition that $f_{k}$ is quasipolynomial. To do this, we will use the following general strategy:

\begin{itemize}
\item Partition the coefficients of $\qbinom{n+k}{k}$ into different classes with periodic behavior.
\item Using the periodicity of the first $\pi_N(k)$ coefficients (by Lemma \ref{27}), inductively show that these sections are also periodic using a partition decomposition (Lemma \ref{34}).
\item Use this last fact to show that $n \mapsto n+\pi'_N(k)$ changes $f_{k}(n+r)$ a constant amount depending only on $r$.
\item Conclude $f_{k}$ is quasipolynomial, since the previous point shows $\mathcal{L}_{k}^{(i)}$ are linear.
\end{itemize}

We begin with the division of coefficients in $\qbinom{n+k}{k}$ into different sections.

\begin{definition}
The $i$th \textit{section} of the $q$-binomial coefficient $\qbinom{n+k}{k}$ is the sequence of coefficients denoted by $\mathsf{S}_i$ with $j$th term given by
\[p_{\le k}^{(i)}(j) = [q^{in+j}]\qbinom{n+k}{k},\]
where $j\in \mathbb{Z}/n\mathbb{Z}$. As a special case, $\mathsf{S}_0$ is just a concatenation of copies of $\mathcal{S}$.
\label{32}
\end{definition}

Recall the identity
\[\sum_{i\ge 0}{p(n,k,i)q^i} = \qbinom{n+k}{k},\]
where $p(n,k,i)$ denotes the number of partitions $\lambda \vdash i$, with at most $k$ parts and maximal part $\le n$. 

This definition allows us to loosely determine a section by saying terms in the sequence contain the number of partitions which fit in a $n\times k$ box of size $|\lambda|=l$ for $l$ such that there exists a partition $\lambda \vdash l$ covering $i$ complete rows but no partition covering $i+1$ rows.

\begin{definition}
\label{33}
Let $X=(x_0, \ldots , x_{|X|-1})$ and $Y=(y_0, \ldots , y_{|Y|-1})$ be finite sequences. The concatenation operator $\oplus$ is defined as $X \oplus Y = (x_0, x_1, \ldots x_{|X|-1}, y_0, y_1, \ldots y_{|Y|-1}).$
\end{definition}
We then make the following decomposition of $\mathsf{S}_i$ that proves useful:
\[\mathsf{S}_i = \mathsf{B}_i^1 \oplus \mathsf{B}_i^2 \oplus \ldots \oplus \mathsf{B}_i^l \oplus \mathsf{R}_i,\]
where the $\mathsf{B}_i^j$ are $\pi'_N(k)$-length subsequences and $\mathsf{R}_i$ is the remainder after these $l=\floor{\frac{n}{\pi'_N(k)}}$ consecutive subsequences are removed from $\mathsf{S}_i$. Informally, if we regard $\qbinom{n+k}{k}$ as a sequence ordered by the associated exponents of $q$, we can relate $X=\bigoplus_{i\in [k]}{\mathsf{S}_{i-1}}\oplus{(1)}$ to its corresponding $q$-binomial coefficient. Here, $(1)$ is just a sequence only containing $1$. We can index $X$ starting at $0$, obtaining
\[\qbinom{n+k}{k} = \sum_{x_i \in X}{x_i q^i}.\]
The net result of this decomposition is illustrated in Figure \ref{fig:sb}.

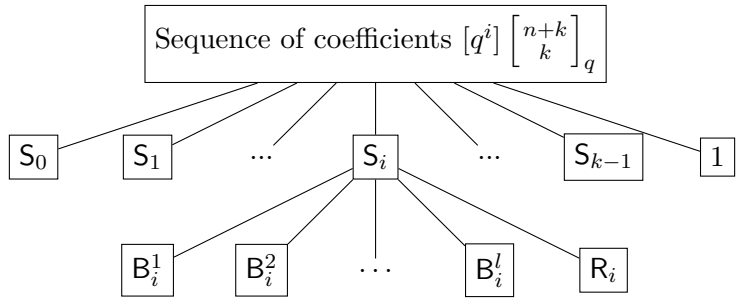
\begin{figure}[ht]
\begin{center}
\begin{tikzpicture}[every node/.style={rectangle,draw}]
\node {Sequence of coefficients $[q^i]\qbinom{n+k}{k}$}
  child {node {$\mathsf{S}_0$} }
  child { node {$\mathsf{S}_1$} }
  child { node[draw=none] {...} }
  child { node {$\mathsf{S}_i$}
    child {node {$\mathsf{B}_i^1$} }
    child {node {$\mathsf{B}_i^2$} }
    child {node[draw=none] {$\ldots$} }
    child {node {$\mathsf{B}_i^l$} } 
    child {node {$\mathsf{R}_i$} } }
  child { node[draw=none] {...} }
  child { node {$\mathsf{S}_{k-1}$} }
  child { node {${1}$} };

\end{tikzpicture}

\caption{Decomposition of a $q$-binomial coefficient into sections modulo $N$. Here, edge connections denote concatenation (as per Definition \ref{33}) from left to right and $l := \floor{\frac{n}{\pi'_N(k)}}$.}
\label{fig:sb}
\end{center}
\end{figure}
\section[Proving the main theorem]{Proving $f_{k}$ is quasipolynomial}

Using the definitions from Section 3, we investigate the structure of each individual section.

\begin{definition} Fix a $q$-binomial coefficient $\qbinom{n+k}{k}$. Let $\mathcal{P}_{i,m}^{\textup{bad}}(j)$ be the set containing all pairs of partitions $(\lambda, \mu)$ such that
\begin{itemize}
\item $|\lambda|+|\mu|=mn+j$.
\item $\lambda$ has at most $k$ parts each at most $n$, of which $i$ are equal to $n$.
\item $\mu$ has exactly $i$ parts.
\end{itemize}
\label{pbad}
\end{definition}

\begin{lemma}
Fix $n,m, k$. For $\qbinom{n+k}{k}$ we have the following identity for the associated functions $p_{\le k}^{(m)}$:
\[p_{\le k}^{(m)}(j) = p_{\le k}(mn+j)- \sum_{i\in [m]}{\#\mathcal{P}^{\textup{bad}}_{i,m}(j)}.\] \label{34}
\end{lemma}
\begin{proof}
Let $S$ be the set of partitions counted by $p_{\le k}^{(m)}(j)$ and $S'$ be defined similarly for $p_{\le k}(mn+j)$. It is clear that $S\subseteq S'$, so we wish to show that $\sum_{i\in [m]}{\#\mathcal{P}^{\textup{bad}}_{i,m}(j)}$ enumerates all of the additional partitions that ``leave the $n\times k$ box" or have some part greater than $n$. Consider Figure \ref{lemma} below.

\begin{figure}[ht]
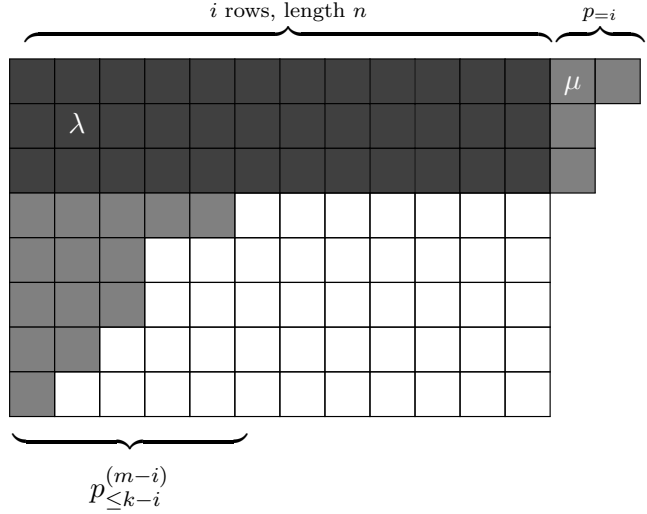

$$\quad \overbrace{\quad \quad \quad \quad \quad \quad \quad \quad \quad \quad \quad \quad \quad \quad \quad \quad\quad \quad }^{i \text{ rows, length }n} \overbrace{\quad \quad \quad }^{p_{=i}} \quad$$
\begin{center}
\ytableausetup{centertableaux}
\begin{ytableau}
*(darkgray) & *(darkgray) & *(darkgray) & *(darkgray) & *(darkgray) & *(darkgray) & *(darkgray) & *(darkgray) & *(darkgray) & *(darkgray) & *(darkgray) & *(darkgray) & *(gray) \textcolor{white}{\mu} & *(gray) \\
*(darkgray) & *(darkgray) \textcolor{white}{\lambda} & *(darkgray) & *(darkgray) & *(darkgray) & *(darkgray) & *(darkgray) & *(darkgray) & *(darkgray) & *(darkgray) & *(darkgray) & *(darkgray) & *(gray)\\
*(darkgray) & *(darkgray) & *(darkgray) & *(darkgray) & *(darkgray) & *(darkgray) & *(darkgray) & *(darkgray) & *(darkgray) & *(darkgray) & *(darkgray) & *(darkgray) & *(gray)\\
*(gray) & *(gray) & *(gray) & *(gray) & *(gray) & & & & & & & \\
*(gray) & *(gray) & *(gray) & & & & & & & & & \\
*(gray) & *(gray) & *(gray) & & & & & & & & & \\
*(gray) & *(gray) & & & & & & & & & & \\
*(gray) & & & & & & & & & & & \\
\end{ytableau}
\end{center}
$$\underbrace{\quad\quad\quad\quad\quad\quad\quad\quad}\quad\quad\quad\quad\quad\quad\quad\quad\quad\quad\quad\quad\quad\quad$$
$$\text{in an } n\times(k-i) \text{ box} \quad\quad\quad\quad\quad\quad\quad\quad\quad\quad\quad\quad\quad$$

\caption{Classification of partitions in $\mathcal{P}^{\textup{bad}}_{i,m}(j)$}
\label{lemma}
\end{figure}

Figure \ref{lemma} depicts a pair $(\lambda, \mu) \in \mathcal{P}^{\textup{bad}}_{i,m}(j)$. Here, $\lambda$ is represented by the shaded boxes inside the $n \times k$ box. The darker boxes depict the $i$ parts of $\lambda$ that are exactly $n$, while the lighter gray boxes below depict the part of $\lambda$ that can vary. Outside of the $n\times k$ boxes is $\mu$, with precisely $i$ parts. As labelled in the diagram, it is easy to see that $\lambda$ are enumerated by $p_{\le k-i}^{(m-i)}$ and $\mu$ are enumerated by $p_{=i}$.
Construct a partition $\Pi = \lambda+\mu$ via part-wise addition. This is counted in $S'$ by $p_{\le k}$ but not in $S$ since it \textit{must} leave the $n\times k$ box. Thus, sending $(\lambda,\mu) \in \bigcup_{i\in [m]}{\mathcal{P}^{\textup{bad}}_{i,m}(j)}$ to $\Pi=\lambda+\mu$ is a map $\phi$ from $\bigcup_{i\in [m]}{\mathcal{P}^{\textup{bad}}_{i,m}(j)}$ to $S'\setminus S$. We claim $\phi$ is a bijection. It is not too difficult to see that $\phi$ is an injection: if $\lambda+\mu = \lambda'+\mu'$ then $\mu$ and $\mu'$ have the same number of parts and from this it is evident $\mu=\mu', \lambda=\lambda'$.

Now we show $\phi$ is a surjection. Take a ``bad" partition $\Pi = \{\pi_1, \ldots , \pi_k\}$ with $|\Pi|=mn+j$ leaving the box. Such a partition must leave the box for the first $i$ rows for some $i\in [m]$ (we cannot have $i>m$, since $|\Pi| = mn+j\le (m+1)n$). Setting $\mu = \{\pi_\alpha - n : \pi_\alpha>n\}$ and $\lambda = \{ \pi_\alpha : \pi_\alpha \le n \} \cup \{n : \pi_\alpha>n\}$, we construct a pair $(\lambda, \mu)$. Both $\lambda,\mu$ satisfy the last two conditions of Definition \ref{pbad} due to the construction. The first condition $|\lambda|+|\mu|=mn+j$ is also satisfied, as
\[|\lambda|+|\mu|=\sum_{\pi_\alpha>n}{(\pi_\alpha-n)+n}+\sum_{\pi_\alpha\le n}{\pi_\alpha}=|\Pi|=mn+j.\]
Thus $(\lambda,\mu) \in \mathcal{P}^{\textup{bad}}_{i,m}(j)$, and $\lambda+\mu=\Pi$.

Thus, we have a bijection $\phi$ from $\bigcup_{i\in [m]}{\mathcal{P}^{\textup{bad}}_{i,m}(j)}$ to $S'\setminus S$, and it follows that 
\[\sum_{i\in [m]}{\#\mathcal{P}^{\textup{bad}}_{i,m}(j)}=|S'\setminus S|.\] \end{proof}

The following is a restatement of Theorem \ref{19} and is the main result. 
\begin{theorem}
Let $N\in \N$ be the modulus of $f_{k}$. Then $f_{k}$ is quasipolynomial, with quasiperiod $\pi_{N}'(k)$. Additionally, all $\mathcal{L}_{k}^{(i)}$ are linear functions. \label{main} \end{theorem} \begin{proof}

Let $Q \eqdef \pi_{N}'(k)$. To prove the claim, the central idea of the argument is to show that in the section $\mathsf{S}_i$ of $\qbinom{Ql+r+k}{k}$ we have $\mathsf{B}_i^1\equiv \mathsf{B}_i^2\equiv \ldots \equiv \mathsf{B}_i^l \pmod{N}$. From this fact, we make a simple argument that shows the claim. We use the $q$-binomial coefficient $\qbinom{n+k}{k}=\qbinom{Ql+k+r}{k}$ so that we can read off that $\mathsf{B}_i$ have length $Q$,  $\mathsf{R}_i$ has length $r$, and that $l$ is the number of $\mathsf{B}_i$ in the decomposition of $\mathsf{S}_i$.

To prove that $\mathsf{B}_i^1\equiv \mathsf{B}_i^2\equiv \ldots \equiv \mathsf{B}_i^l \pmod{N}$, we induct on the indices $m=1, 2, \ldots, k-1$ of the sections, holding $k$ fixed, and then induct on $k$.

In \S 2, we already showed that $\mathsf{S}_0$ has the aforementioned property for all $k$ by considering partitions with at most $k$ parts. One can also show that $\mathsf{B}_i^1=\mathsf{B}_i^2=\ldots = \mathsf{B}_i^l$ holds when $k=2$ by explicit computation of $p_{\le 2}(j)=\floor{\frac{j}{2}}+1.$ This establishes the base cases $m=*, k=2$ and $m=0, k=*$. 

We show the claim holds for $\mathsf{S}_m$ assuming it holds for $\mathsf{S}_{i}$ ($i<m$) and all smaller $k$. Using Lemma \ref{34}, we have for $j \in \mathbb{Z}/(Ql+r)\mathbb{Z}$,
\begin{align*}
p_{\le k}^{(m)}(j) &\equiv p_{\le k}(j+mr)- \sum_{i\in [m]}{\#\mathcal{P}^{\textup{bad}}_{i,m}(j)} \\
&\equiv p_{\le k}(j+mr)-\sum_{i\in[m]}{\sum_{\ell+\ell' = C_{i,m}(j)}{[q^\ell]\qbinom{n+(k-i)}{k-i}p_{=i}(\ell')}} \pmod{N} \tag{3} \\
\end{align*}
where $C_{i,m}(j) = |\lambda|+|\mu|-ni$ for $(\lambda, \mu) \in \mathcal{P}^{\textup{bad}}_{i,m}(j)$ (or more explicitly $C_{i,m}(j)=(m-i)(Ql+r)+j$) and $\ell, \ell' \ge 0$. The functions $p_{=i}$ and $[q^\ell]\qbinom{n+(k-i)}{k-i}$ count $\mu$ and $\lambda$ in $\mathcal{P}^{\textup{bad}}_{i,m}$ respectively in (3). Note that
\[p_{=i}(\ell) = p_{\le i}(\ell-i),\]
an explicit bijection being given by taking $\tau \vdash n$ counted by $p_{=i}$ and decreasing each part by one. Thus, we see $p_{=i}$ has period $Q$ as $\pi_N(i) | \pi_N(k) | Q$. We claim that the map $j \mapsto j+Q$ leaves $p_{\le k}^{(m)}(j)$ unchanged modulo $N$, or that $p_{\le k}^{(m)}(j+Q)-p_{\le k}^{(m)}(j) \equiv 0 \pmod{N}$. Note here that $j+Q<n$, otherwise this statement does not make sense (so for example, $l \ge 1$ so we have a complete block $\mathsf{B}_i$). Since $p_{\le k}$ has period $\pi_N(k) \mid \pi'_N(k)$, the function $p_{\le k}$ will vanish in the difference. So it suffices to show that $\Delta:=\#\mathcal{P}^{\textup{bad}}_{i,m}(j+Q)-\#\mathcal{P}^{\textup{bad}}_{i,m}(j)\equiv 0 \pmod{N}$.

\begin{align*}\Delta &\equiv {\sum_{\ell+\ell' = C_{i,m}(j+Q)}{[q^\ell]\qbinom{n+(k-i)}{k-i}p_{=i}(\ell')}}-{\sum_{\ell+\ell' = C_{i,m}(j)}{[q^\ell]\qbinom{n+(k-i)}{k-i}p_{=i}(\ell')}}\\
&\equiv \sum_{\substack{\ell+\ell' \equiv j \\ (\textup{mod } Q)}}{p_{\le k-i}^{(m-i)}(\ell)p_{=i}(\ell')}. \pmod{N} \tag{4}
\end{align*}
The final congruence requires some elaboration. Here, we observe that $\ell$ determines $\ell'$ entirely, and that $C_{i,m}(j+Q)-C_{i,m}(j)=Q$. Since $p_{=i}(\ell')$ has period $Q$ modulo $N$, we can add $Q$ to $\ell'$ to eliminate all terms except for $C_{i,m}(j)< \ell \le C_{i,m}(j+Q)$. The exact bounds for $\ell$ are unimportant, since $p_{=i}(0)=0$ and hence we can ignore terms where $\ell'\equiv 0 \pmod{Q}$ in the sum - thus, equivalently we have $C_{i,m}(j)\le \ell \le C_{i,m}(j+Q)$ or $C_{i,m}(j)\le \ell < C_{i,m}(j+Q)$. Hence, this corresponds to $p_{\le k-i}^{(m-i)}(\overline{\ell}):=[q^{(m-i)n+\overline{\ell}}]\qbinom{n+(k-i)}{k-i}$ for $\overline{\ell}\in \Z/Q\Z$. Note that since $l\ge 1$, this stays within bounds for $\ell$ for $p_{\le k-i}^{(m-i)}(\ell)$.

In the final sum in (4), $\ell, \ell' \in \Z/Q\Z$. It follows immediately from the definition of $\pi'_N(k)$ that $\pi_{N}'(k)/\pi_{N}'(k-1) \in N\mathbb{Z}$. But note that $p_{\le k-i}^{(m-i)}$ and $p_{=i}$ have periods dividing $\pi_{N}'(k-1)$ as it is always true that $1\le i\le k-1$ for each sum, and hence residues modulo $N$ are repeated some multiple of $N$ times in the sum. Thus $p_{\le k}^{(m)}(j)$ is $Q$-periodic since the sum is $0$ modulo $N$, and by strong induction the same is true for each $\mathsf{S}_m$. This completes the induction.

Then $l \mapsto l+1$ simply adds on another identical period in each $\mathsf{S}_m$. Hence, we may write
\[\mathsf{S}_i = \mathsf{B}_i \oplus \mathsf{B}_i \oplus \ldots \oplus \mathsf{B}_i \oplus \mathsf{R}_i,\]
where the $\mathsf{B}_i$ are identical modulo $N$. For short, denote this $\mathsf{S}_i = \mathsf{B}_i^{\oplus l}\oplus\mathsf{R}_i$. More importantly, this indicates that  $f_{k}(Q(l+1)+r)-f_{k}(Ql+r)$ is a constant depending on $r$. Thus, we can write
\[f_{k}(n)=\mathcal{L}_{k}^{(i)}\left(\frac{n-i}{Q}\right),\]
which is precisely what we wanted.
\end{proof}

The decomposition used in the Theorem \ref{main} also allows us to prove the following observation about a special case of $p_{\le k}^{(m)}$:

\begin{corollary}
The last $\binom{k+1-m}{2}-1$ entries of each component $\mathsf{B}_m^i$ of $\mathsf{S}_m$ are $0$ modulo $N$ for $\qbinom{\pi'_N(k)l+k}{k}$. 
\end{corollary}

\begin{proof}
This is given for $m=0$ by Corollary \ref{28}, so let $m>0$. Similarly, $k=2$ is trivial. We proceed by strong induction on $k,m$. By Lemma \ref{34}, for $j \in \mathbb{Z}/\pi'_N(k)l\mathbb{Z}$ we have

\[p_{\le k}^{(m)}(j) \equiv p_{\le k}(j)- \sum_{i\in[m]}{\#\mathcal{P}_{i,m}^{\textup{bad}}}(j) \pmod{N}.\]
Noting that $\binom{a}{2} < \binom{b}{2}$ when $a<b$ we see for $j\in[\pi'_N(k) - \binom{k+1-m}{2}, \pi'_N(k)-1]$ that $p_{\le k}(j)\equiv 0 \pmod{N}$ by Corollary \ref{28}. Therefore, for such $j$ we have the simplified form
$p_{\le k}^{(m)}(j) \equiv -\sum_{i\in [m]}{\#\mathcal{P}_{i,m}^{\textup{bad}}(j)} \pmod{N}$. We wish to show $\sum_{i\in [m]}{\#\mathcal{P}_{i,m}^{\textup{bad}}(j)} \equiv 0 \pmod{N}$ for such $j$. To do this, we use the expansion of $\#\mathcal{P}_{i,m}^{\textup{bad}}(j)$ from the main theorem and exploit that $\#\mathcal{P}^{\textup{bad}}_{i,m}(j+\pi'_N(k))\equiv \#\mathcal{P}^{\textup{bad}}_{i,m}(j) \pmod{N}$ to obtain
\begin{align*}
\#\mathcal{P}_{i,m}^{\textup{bad}}(j) &\equiv \sum_{\ell+\ell'=j}{[q^\ell]\qbinom{n+(k-i)}{k-i}p_{=i}(\ell')} \pmod{N} \\
&\equiv \sum_{\substack{\ell+\ell' \equiv j \\ (\textup{mod }\pi'_N(k))}}{p_{\le k-i}^{(0)}(\ell)p_{=i}(\ell')} \pmod{N}
\end{align*}
The final congruence is because the last $\binom{(k-i)+1}{2}-1 \ge \binom{k+1-m}{2}-1$ entries of $[q^\ell]\qbinom{n+(k-i)}{k-i}$ are $0$ by Corollary \ref{28} and Lemma \ref{21} as $\ell < \pi'_N(k)$. All added terms in the summation will have $\ell, \ell'$ lie in the interval $[j, \pi'_N(k)-1]$ and our specifically chosen $j$ makes it so that each new term added must then be $0$ mod $N$. The final sum is $0$ modulo $N$ for the same reasons as in the main theorem, since $k-1\ge i\ge 1$ and hence as a function of $\ell$ the function $p_{\le k-i}^{(0)}(\ell)p_{=i}(\ell')$ has a period $P$ so $\pi'_N(k)/P\in N\Z$ by definition of $\pi'_N(k)$.
\end{proof}

\begin{remark}
Using the symmetry of the $q$-binomial coefficients, we can show that a similar claim holds for the first entries of $\mathsf{B}_m^i$. Explicitly, the first $\binom{k+1-((k-1)-m)}{2}-1=\binom{m+2}{2}-1$ entries are $0$ modulo $N$.
\end{remark}

\section[The generating function]{The generating function of $f_{k}$}

The result from the previous section allows for the generating function for $f_{k}$ to be explicitly calculated.

\begin{theorem}
For a modulus $N \in \N$, we have
\[F_{k}(x) \eqdef \sum_{n\ge 0}{f_{k}(n)x^n} = \frac{1}{(1-x^{Q})^2}\sum_{i \in \Z/Q\Z}{(1-x^{Q})b_ix^i+m_ix^{Q+i}},\]
where $\mathcal{L}_{k}^{(i)}$ has constant term $b_i$ and slope $m_i$ and $Q=\pi'_N(k)$.
\label{41}
\end{theorem}
\begin{proof}
For simplicity, let $\mathcal{L}_i =\mathcal{L}_{k}^{(i)}$, and $Q=\pi'_N(k)$ as above. Then we let $F_{\mathcal{L}_i}(x) = \sum_{j\ge 0}{\mathcal{L}_i(j)x^{jQ + i}},$ so that
\[F_{k}(x) = \sum_{i\in \Z/Q\Z}{F_{\mathcal{L}_i}(x)}.\]
Fortunately, each term is simple to find. We have
\begin{align*}
\sum_{i\in \Z/Q\Z}{F_{\mathcal{L}_i}(x)} &= \sum_{i\in \Z/Q\Z}{x^i \left(\sum_{j\ge 0}{b_ix^{jQ}}+\sum_{j\ge 0}{m_i j x^{jQ}} \right)}\\
&= \sum_{i\in \Z/Q\Z}{x^i \left(\frac{b_i}{1-x^{Q}} + \frac{m_ix^{Q}}{(1-x^{Q})^2}\right)} \\ &=\frac{1}{(1-x^Q)^2}\sum_{i\in \Z/Q\Z}{(1-x^{Q})b_ix^i+m_ix^{Q+i}}
\end{align*}
which proves the theorem.

\end{proof}

Letting $Q'=\pi'_N(k-1)$, it turns out that one can often rewrite this as 
\[F_{k}(x)=\frac{1}{(1-x^{Q})^2}\left(\sum_{i\in \Z/Q\Z}{(1-x^{Q})b_ix^i+\sum_{i\in \Z/Q'\Z}{m_ix^{Q+i}\frac{x^{Q}-1}{x^{Q'}-1}}}\right).\]
This stems from the fact that the slopes of the functions $\mathcal{L}_{k}^{(i)}$ often have a smaller period (in $i$, where $k,R$ are fixed) than the actual quasiperiod itself, namely $Q'$. This is formalized by Theorem \ref{period}, and an example is given in Figure \ref{illustration}.

\begin{figure}[!ht]
\begin{center}
\includegraphics[scale=0.7]{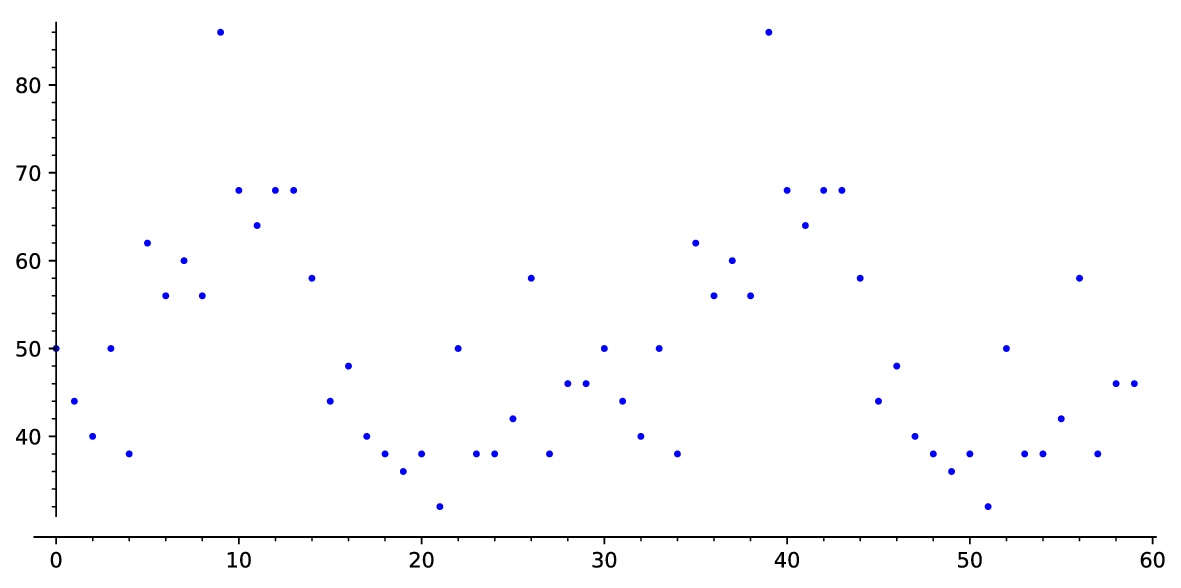}
\end{center}
\caption{\textbf{An illustration of Theorem \ref{period}.} This figure shows the slopes of the functions $\mathcal{L}_{4}^{(i)}$ when $N=5, R=1$. The horizontal axis is $i$, while the vertical axis is the slope. Notice that the slopes have a period that is half of the actual minimal quasiperiod (in this case, given by the function $\pi_5$) and $\pi_5(4)/\pi_5(3)=60/30=2,$ as claimed.}
\label{illustration}
\end{figure}

\begin{theorem}
Let $\qbinom{n+k}{k}$ be a $q$-binomial coefficient where $n>\pi'_N(k)$ so that we have an entire block $\mathsf{B}_i$ in each section. Fix a pair $(k,N)$ such that $\sum_{i\in \Z/\pi'_N(k)\Z}{p_{\le k}(i)}\equiv 0 \pmod{N}$ by Theorem \ref{29}. Then $\sum_{i\in \Z/\pi'_N(k)\Z}{p^{(m)}_{\le k}(i)}\equiv 0 \pmod{N}$.
\label{0sum}
\end{theorem}

\begin{proof}
We use Lemma \ref{34}. This yields, following the same expansion as in Theorem \ref{main},
\[p_{\le k}^{(m)}(j) \equiv p_{\le k}(j+mn)-\sum_{i\in[m]}\left({\sum_{\ell+\ell' = C_{i,m}(j)}{[q^\ell]\qbinom{n+(k-i)}{k-i}p_{=i}(\ell')}}\right) \pmod{N}.\]
Summing over $j\in \Z/\pi'_N(k)\Z$, by assumption the $p_{\le k}$ term disappears since $p_{\le k}$ has period $\pi_N(k)$. It suffices to show that
\[\sum_{j\in \Z/\pi'_N(k)\Z} {\sum_{\ell+\ell' = C_{i,m}(j)}{[q^\ell]\qbinom{n+(k-i)}{k-i}p_{=i}(\ell')}} \equiv 0 \pmod{N}.\]
We can equivalently sum over $\ell$ and then $\ell'$ so that $\ell + \ell' \in [C_{i,m}(0),C_{i,m}(\pi'_N(k)-1)]$. Suppose that $\ell \le C_{i,m}(0)$. Observe that \[\sum_{j\in \Z/\pi'_N(k)\Z}{[q^\ell]\qbinom{n+(k-i)}{k-i}p_{=i}(\ell'+j)}\equiv 0 \pmod{N},\]
since $[q^\ell]\qbinom{n+(k-i)}{k-i}$ is fixed and $p_{=i}(\ell)$ is $p_{\le i}(\ell-i)$ and so this sum becomes $0$ by our assumption. 

We need to show that the sum of all terms $\ell>C_{i,m}(0)$ is zero. Our reformulation gives us a sum over $\ell$ and $\ell'$ so $\ell+\ell' \le C_{i,m}(\pi'_N(k)-1)$ and $\ell>C_{i,m}(0)$. Although we can ignore $\ell=C_{i,m}(0)$ because it have no contribution, we include it in the following calculation to make it simpler. We can observe that by the restrictions on $n$ that $[q^{\ell}]\qbinom{n+(k-i)}{k-i}$ stays in the same section. Let $Q=\max(\pi'_N(k-i),\pi'_N(i))$. Each possible pairing $([q^{\ell}]\qbinom{n+(k-i)}{k-i}, p_{=i}(\ell'))$ modulo $N$ is repeated a multiple of $\sum_{x\in[\pi'_N(k)/Q]}{x}=\frac{1}{2}(\pi'_N(k)/Q)(\pi'_N(k)/Q+1)$ times due to their respective periods. But since $N$ divides $\pi'_N(k)/Q$ ($i\le m \le k-1$) and $N$ is odd (the pairs in Theorem \ref{29} have odd $N$), this is a multiple of $N$ and hence the entire sum becomes $0$ modulo $N$.
\end{proof}

\begin{theorem}
Fix a pair $(k,N)$ such that
\[\sum_{i\in \Z/\pi'_N(k-1)\Z}{p_{\le k-1}(i)}\equiv 0. \pmod{N}\]
by Theorem \ref{29}. Then the slope of $\mathcal{L}_{k}^{(i)}$ is equal to that of $\mathcal{L}_{k}^{(i')}$ where $i' \equiv i+\pi'_N(k-1) \pmod{\pi'_N(k)}$.
\label{period} 
\end{theorem}

\begin{proof}
In order to prove the theorem, we actually make a deeper claim. Consider the $q$-binomial coefficients $\qbinom{n+k}{k}, \qbinom{\widetilde{n}+k}{k}$ where $\widetilde{n} = n+\pi'_N(k-1)$ and decompose the coefficients into $\mathsf{S}_i$ and $\widetilde{\mathsf{S}}_i$ respectively. Then when we make the decompositions $\mathsf{S}_i = \mathsf{B}_i^{\oplus l} \oplus \mathsf{R}_i$ and $\widetilde{\mathsf{S}}_i = \widetilde{\mathsf{B}}_i^{\oplus l} \oplus \widetilde{\mathsf{R}}_i$, we want to show that $\mathsf{B}_i$ is a cyclic shift of $\widetilde{\mathsf{B}}_i$. Since the main theorem implies that $\mathsf{B}_i$ is determined by the residue class of $n \pmod{\pi'_N(k)}$, it suffices to show this for $n$ sufficiently large.

Using Lemma \ref{34}, we have 
\[p_{\le k}^{(m)}(j) \equiv p_{\le k}(j+mn)-\sum_{i\in[m]}\left({\sum_{\ell+\ell' = C_{i,m}(j)}{[q^\ell]\qbinom{n+(k-i)}{k-i}p_{=i}(\ell')}}\right) \pmod{N}\]
where $C_{i,m}(j) = (m-i)n+j$. Now we take $n \mapsto \widetilde{n}=n+\pi'_N(k-1)$ and obtain a function $\widetilde{p}_{\le k}^{(m)}(j)$ for $\widetilde{\mathsf{S}}_m$. If we take $j \mapsto \widetilde{j}=j+m\pi'_N(k-1)$, we claim that
\[p_{\le k}^{(m)}(\widetilde{j}) \equiv  \widetilde{p}^{(m)}_{\le k}(j) \pmod{N}.\]
This is equivalent to $\mathsf{B}_i$ being a cyclic shift of $\widetilde{\mathsf{B}}_i$. Here, we are implicitly taking $\tilde{j} \pmod{\pi'_N(k)}$ and treating $p_{\le k}^{(m)}$ as a periodic sequence - this avoids confusion with $\widetilde{\mathsf{R}}_i$, as the shift is within each block $\mathsf{B}_i \mapsto \widetilde{\mathsf{B}}_i$. Using Lemma \ref{34} again, we see this is equivalent to
\[p_{\le k}(\widetilde{j}+mn)-\sum_{i\in[m]}{\#\mathcal{P}^{\textup{bad}}_{i,m}(\widetilde{j})} \equiv p_{\le k}(j+m\widetilde{n})-\sum_{i\in[m]}{\#\mathcal{P}^{\textup{bad}}_{i,m}(j)} \pmod{N}.\]

As $\widetilde{j}+mn=j+m\widetilde{n}$, we need only consider the sums $\sum_{i\in [m]}{\#\mathcal{P}^{\textup{bad}}_{i,m}(\cdot)}$. We want to show that this is invariant modulo $N$ under $j \mapsto \widetilde{j}$ and $n\mapsto \widetilde{n}$. We get for individual terms (indexed by $i$) the difference
\[ \sum_{\ell+\ell' = {C}_{i,m}(\widetilde{j})}{[q^\ell]\qbinom{n+(k-i)}{k-i}p_{=i}(\ell')} - \sum_{\ell+\ell' = \widetilde{C}_{i,m}(j)}{[q^\ell]\qbinom{n+(k-i)}{k-i}p_{=i}(\ell')} \pmod{N} \]
where $\widetilde{C}_{i,m}(j) = (m-i)\widetilde{n}+j$, differing from $C_{i,m}(j)$ by a multiple of $\pi'_N(k-1)$. We claim that this is $0$ for $i>1$. Because $p_{=i}$ has period $\pi_N(i) \mid \pi'_N(k-1)$ and $\pi'_N(k-1)\mid {C}_{i,m}(\widetilde{j})-\widetilde{C}_{i,m}(j)$, we can add this quantity to $\ell'$ in the second sum (leaving it unchanged) to cancel terms. After this, we are left with
\[\sum_{\substack{\ell + \ell' = {C}_{i,m}(\widetilde{j}) \\ \ell > \widetilde{C}_{i,m}(j)}}{[q^\ell]\qbinom{n+(k-i)}{k-i}p_{=i}(\ell')} \pmod{N}.\]
The difference is ${C}_{i,m}(\widetilde{j})-\widetilde{C}_{i,m}(j)=i\pi'_N(k-1)$. For $k-1>i>1$, recall the division of residues in $\qbinom{n+(k-i)}{k-i} \pmod{N}$ - the blocks have size $\pi'_N(k-i) = \text{len}(\mathsf{B_\bullet})$ in each section. For $n$ sufficiently large, we can keep all $\ell$ in the same section as $n=\text{len}(\mathsf{S}_\bullet)$ and $i\pi'_N(k-1)$ is independent of $n$. Note also that $\pi'_N(i)$ is a period of $p_{=i}$. As $\pi'_N(k-1)/\pi'_N(k-i), \pi'_N(k-1)/\pi'_N(i)\in N\Z$ for $1<i<k-1$, this implies the entire sum is $0$ because we repeat each period a multiple of $N$ times and either $\pi'_N(i) \mid \pi'_N(k-i)$ or the other way around. At $i=k-1$, the $q$-binomial coefficient has all coefficients $1$ so this becomes $0$ by the restrictions on $(k,N)$. For $i=1$, $p_{=1}(\ell')=1$ and the sum is $0$ again by the restrictions on $(k,N)$ and Theorem \ref{0sum}, since $n$ is large enough that the $\ell$ values stay in the same section and we just sum over a period $i$ times.

Thus, for sufficiently large $n$ we have $p_{\le k}^{(m)}(\widetilde{j}) = \widetilde{p}^{(m)}_{\le k}(j)$. We conclude that $\widetilde{\mathsf{B}}_i$ is a cyclic shift of $\mathsf{B}_i$ and the result follows since the slopes of $\mathcal{L}_{k}^{(i)}$ depend only on the number of occurrences of the residue $R$ in each $\mathsf{B}_i$ (which clearly is the same under a cyclic shift).
\end{proof}

\section{Asymptotics for the quasiperiod}
Given the complex nature of the definition for $\pi'_N(k)$ it is worth investigating asymptotics to understand how quickly $f_{k}(n)$ and its generating function grow in complexity. First we investigate asymptotics for $\pi_p(k)$ for each prime $p$. We have the expansion
\[\pi_p(k) = p^{b_p([k])} L_p([k]),\]
where $b_p([k])$ and $L_p([k])$ are as previously defined in Theorem \ref{22}. Note that $\text{lcm}([k]) = e^{\psi(k)}$ where $\psi(k)$ is the Chebyshev function. Let \[\Pi(k) \eqdef \sum_{i\in[k]}{p^{\nu_p(i)}}.\]
We first consider the asymptotics of this function in Lemma \ref{51}.

\begin{lemma}\[\Pi(k) = \sum_{i\in[k]}{p^{\nu_p(i)}} \sim \frac{p-1}{p}k\log_p(k).\]
\label{51}
\end{lemma}
\begin{proof}
This can be done by observing that $\sum_{i\in[k]}{p^{\nu_p(i)}} = \sum_{i=0}^{\lfloor \log_p(k) \rfloor}{\#V_{i,p} p^i}$, where $V_{i,p} = \{ j | j\in[k],  \nu_p(j) = i\}$. Now we can take $\#V_{i,p} = \left\lfloor \frac{k}{p^i}\right\rfloor - \left\lfloor \frac{k}{p^{i+1}} \right\rfloor$, yielding
\[\sum_{i=0}^{\lfloor \log_p(k) \rfloor}{\#V_{i,p} p^i} = \sum_{i\ge 0}{ \left( \left\lfloor \frac{k}{p^i}\right\rfloor - \left\lfloor \frac{k}{p^{i+1}} \right\rfloor \right)p^i} \sim  (p-1) \left(\sum_{i\ge 1}{\left\lfloor \frac{k}{p^i} \right\rfloor p^{i-1}}\right).\]
We now want to show $\sum_{i\ge 1}{\left\lfloor \frac{k}{p^i} \right\rfloor p^{i-1}} \sim \frac{k\log_p(k)}{p}$ for large $k$. That is, we want to show that $\lim_{k \to \infty}{\frac{k\log_p(k)}{\sum_{i\ge 1}{\left\lfloor \frac{k}{p^i} \right\rfloor p^{i}}}} = 1$. We obtain upper and lower bounds for the limit via $\frac{k}{p^i}\ge \lfloor \frac{k}{p^i} \rfloor \ge \frac{k}{p^i}-1$, yielding the bounds
\[k\log_p(k) \ge \sum_{i\ge 1}{\left\lfloor \frac{k}{p^i} \right\rfloor p^{i}} \ge k\bra{\log_p(k)-1-\frac{p}{p-1}}. \]
Upon dividing we see
\[\frac{k\log_p(k)}{\sum_{i\ge 1}{\left\lfloor \frac{k}{p^i} \right\rfloor p^{i}}} \le \frac{k\log_p(k)}{k(\log_p(k)-1-\frac{p}{p-1})} = 1+\frac{(1+\frac{p}{p-1})k}{k\log_p(k)-(1+\frac{p}{p-1})k}.\]
Thus, the limit is bounded above by 1. The lower bound clearly goes to 1, and we conclude that
\[\sum_{i\in[k]}{p^{\nu_p(i)}} \sim \frac{p-1}{p}k\log_p(k).\] 
\end{proof}

\noindent The following lemma will be useful in understanding $L_p([k])$.

\begin{lemma}
$\nu_p(\textup{lcm}([k])) = \lfloor \log_p(k)\rfloor$.
\label{52}
\end{lemma}

\noindent Now we can make an asymptotic analysis for the log of $\pi_p(k)$, since we have asymptotics relating to both components of $\pi_p(k)$.

\begin{theorem} We have
\begin{align*}
\log_p(\pi_p(k)) &\sim \log_p\log_p(k) + \frac{\psi(k)}{\ln p}.
\end{align*}
\label{53}
\end{theorem}

\begin{proof}
Using Lemma \ref{52}, we see
\[\pi_p(k) = p^{b_p([k])}L_p([k]) = \text{lcm}([k])p^{b_p([k]) - \lfloor \log_p(k) \rfloor}.\]
This can be simplified to 

\begin{align*} 
\log_p\pi_p(k) &= \frac{\psi(k)}{\ln p}+\bra{b_p([k]) - \lfloor \log_p(k) \rfloor} \\
&=\frac{\psi(k)}{\ln p}+\bra{\lceil \log_p\Pi(k) \rceil - \lfloor \log_p(k) \rfloor}.  \tag{5}
\end{align*}
We use Lemma \ref{51} to show $\log_p\Pi(k) \sim \log_p\left(\frac{p-1}{p}k\log_p(k) \right)$. Ignoring constant terms, we simplify this asymptotic to
\[\log_p\Pi(k) \sim \log_p\log_p(k)+\log_p(k),\]
and in the limit floors become irrelevant so the $\log_p(k)$ term is cancelled in (5), yielding the desired asymptotic.
\end{proof}

\begin{lemma}
$\log_p\bra{\frac{\pi_p'(k)}{\pi_p(k)}} \sim k-\log_p(k)-\log_p\log_p(k)$
\label{54}
\end{lemma}

\begin{proof}
Note that
\[\frac{\pi'_p(k)}{\pi_p(k)} = p^{\#S_k},\]
where
\[S_k = \left\{i : i\le k, \frac{\pi_p(i)}{\pi_p(i-1)}\not \in p\mathbb{Z}\right\}.\]
The condition in $S_k$ can be re-written in terms of the p-adic valuation as $\nu_p(\pi_p(i)) = \nu_p(\pi_p(i-1))$. But this valuation is just $b_p([i])$, so we really have $\#S_k = k-b_p([k])$. Now we can write
\[\#S_k = k - \left\lceil \log_p \left( \Pi(k) \right)\right\rceil\]
where we already have an asymptotic formula for $\Pi(k)$. We can obtain
\[\#S_k \sim k - \log_p\left(\frac{p-1}{p}k\log_p(k)\right).\]
This implies that
\begin{align*}
\log_p \left( \frac{\pi'_p(k)}{\pi_p(k)} \right) &= \#S_k \\ 
&\sim k - \log_p\left(\frac{p-1}{p}k\log_p(k)\right).
\end{align*}
Ignoring constants in the above asymptotic, we obtain the desired asymptotic.
\end{proof}

By understanding $\pi_p$, we can easily derive formulas for $\pi'_p$ by simply accounting for a power of $p$ as above. That is,
\begin{align*}
\log_p \pi'_p(k) &\sim \log_p \pi_p(k) + k - \log_p(k) - \log_p \log_p(k)\\
&\sim \frac{\psi(k)}{\ln p}+k - \log_p(k),
\end{align*}
where we have already bounded $\pi_p$ via Theorem \ref{53}. The next task is to understand $\pi'_{p^e}(k)$. The lemma below yields an asymptotic estimate for $\log_p \pi'_{p^e}(k)$ through Theorem \ref{53} and Lemma \ref{54}.

\begin{lemma}
For prime powers of an odd prime, we have the formula 
\[\pi'_{p^e}(k) = p^{(k-1)(e-1)}\pi'_p(k). \tag{6}\]
For $p=2$, this is off by a constant factor of $\frac{1}{2}$ if $e>1$.
\end{lemma} 

\begin{proof} The idea behind this is to show $\nu_p(\pi_p(j+1)/\pi_p(j))\le 1$ for $j\ge 2$. This follows from $b_p([j+1])-b_p([j])\le 1$, so we prove this claim instead. Observe that $\Pi(j+1)-\Pi(j)\le j+1$. Let $p^i \le j$ be the largest prime power of $p$ less than or equal to $j$. But then 
\[\Pi(j+1) \le p^{b_p([j])} + j+1 \le p^{b_p([j])}+p^{i+1} \le p^{b_p([j])+1}.\]
Note that $b_p([j]) \ge i+1$, since $j\ge 2$ implies $\Pi(j)>p^i$. Then we conclude the final inequality from $p^{b_p([j])}+p^{i+1} = (1+p^{-\ell})p^{b_p([j])}$ for some $\ell \ge 0$, and $p\ge 1+p^{-\ell}$ as $p\ge 2$. It follows that $b_p([j+1])-b_p([j])\le 1$ and $\nu_p(\pi_p(j+1)/\pi_p(j))\le 1$ for $j\ge 2$.

Suppose we know $\nu_p(\pi_p(j+1)/\pi_p(j))\le 1$ at $j$. First consider the case when $\nu_p(\pi_p(j+1)/\pi_p(j)) = 1$. By definition 
\[\pi'_{p^e}(j+1) = \pi'_{p^e}(j) \lcm \left(p^e, \frac{\pi_p(j+1)}{\pi_p(j)}\right)\]
since $\pi_p(j+1)/\pi_p(j) = \pi_{p^e}(j+1)/\pi_{p^e}(j)$. In this case, to obtain $\pi'_{p^e}(j+1)$ we scale $\pi'_{p^e}(j)$ by $p^{e-1}$ in addition to scaling by $\frac{\pi_p(j+1)}{\pi_p(j)}$. In the case where $\nu_p(\pi_p(j+1)/\pi_p(j)) = 0$, the very same reasoning shows we scale $\pi'_{p^e}(j)$ by $p^{e}$ in addition to scaling by $\frac{\pi_p(j+1)}{\pi_p(j)}$. 

Using this, we can obtain the exact formula for $\pi'_{p^e}(k)$ recursively if we assume $\nu_p(\pi_p(j+1)/\pi_p(j))\le 1$ holds for $1\le j < k$ (we proved it for $j\ge 2$). In $k-1$ recursive steps, we arrive at $\pi'_{p^e}(1):=1$. Using the results above, we obtain 
\[\pi'_{p^e}(k) = \prod_{1\le j<k} p^{e-\nu_p\left(\frac{\pi_p(j+1)}{\pi_p(j)}\right)} \frac{\pi_p(j+1)}{\pi_p(j)}\]
by compactly summarizing the factors across the two cases $\nu_p\left(\frac{\pi_p(j+1)}{\pi_p(j)}\right)=1$ and $0$. Since these are the only possibilities, when we pull out $p^{(e-1)(k-1)}$ we see that the factor which remains is $p^\ell (\prod_{1\le j < k}\frac{\pi_p(j+1)}{\pi_p(j)})$ where $\ell = \#\{j: 1\le j < k, \nu_p\left( \frac{\pi_p(j+1)}{\pi_p(j)}\right)=0\}$ is the number of times we have $p^e$ instead of $p^{e-1}$. Applying the same recursive computation in the $e=1$ case for $\pi'_p(k)$, this is $\pi'_p(k)$. Hence, assuming $\nu_p(\pi_p(j+1)/\pi_p(j))\le 1$ for all steps we obtain $\pi'_{p^e}(k) = p^{(k-1)(e-1)}\pi'_p(k)$.

For $p\neq 2$, we encounter no issues because we still have $\nu_p(\pi_p(2)/\pi_p(1))\le 1$, and so the exact formula holds. This is because $b_p([2])-b_p([1])=1-0=1$. For $p=2$ this does not hold, but our description of the recursive steps holds for all but the $j=1$ step where it is off by a constant factor, and so for $p=2$ the formula is off by a constant factor which is straighforward to compute.
\end{proof}

A useful consequence of this lemma is that
\[\nu_q(\pi'_{p^e}(k)) \sim \nu_q(p^{(k-1)(e-1)} \pi'_p(k))\]
for any prime $q$ since for odd primes these are equal and for $p=2$ they differ by a constant. Similarly, replacing $\nu_q$ with $\log_q$ or $\ln$ we have asymptotics for the logs.

We wish to use this result to understand the growth of $\ln \pi'_N(k)$. We can obtain a loose upper bound on $\ln \pi'_N(k) = \ln \lcm_{p\mid N}{\pi'_{p^{\nu_p(N)}}(k)}$ by taking the product of these asymptotics. However, we can do better. As $k\to \infty$, we simply need to identify which term has the largest power of $p$ associated to it for a given $p$. Set $\nu_q(N) = e_q$. For a prime $q\mid N$ not equal to $p$, we get \[\nu_p(\pi'_{q^{e_q}}(k)) \sim \nu_p\left(q^{(k-1)(e_q-1)} \pi'_q(k)\right)=\nu_p(\pi'_q(k)).\]
The last equality is because $\nu_p(q)=0$. We have $\nu_p(\pi'_q(k)) \sim \log_p(k)$ by Lemma \ref{52}. For $p$, we get
\[\nu_p\left(p^{(k-1)(e_p-1)}\pi'_p(k)\right)=\nu_p(\pi'_p(k)) + (k-1)(e_p-1).\]
Asymptotically, it is clear that this will quickly become dominant. We have
\[\nu_p \bra{\pi'_p(k)} \sim k-\log_p(k)-\log_p \log_p(k) + \log_p \Pi(k)\]
by Lemma \ref{54}, $\nu_p(\pi_p(k))=b_p([k])\sim \log_p \Pi_p(k)$, and that $\pi'_p(k)/\pi_p(k)$ is a power of $p$. Thus, using Lemma \ref{51} we see that only the $k$ is relevant for the asymptotic, so for $p\mid N$ we have $\nu_p(\pi'_N(k)) \sim \nu_p(\pi'_{p^{e_p}}(k)) \sim ke_p$. For primes $p\nmid N$, we get an overall contribution to $\ln \pi'_N(k)$ in the limit of $\ln(\lcm([k]))-\sum_{p \mid N}\ln k$, where the sum comes from Lemma \ref{52}. Putting these together, since the logs in the sum are dominated by other terms we have
\begin{align*}\ln \pi'_N(k) &\sim \sum_{p\mid N}{k e_p \ln p} + \psi(k) \\
&= k\ln N + \psi(k).
\end{align*}
Thus, we have proven the following theorem:

\begin{theorem}
We have
\[\ln \pi'_N(k) \sim k\ln N + \psi(k).\]
\end{theorem}

\section{Conclusion and future directions}

We have shown that the function $f_{k}(n)$ is quasipolynomial modulo any $N \in \N$, from which an explicit formula for the generating function $F_{k}(x)$ follows. Additionally, the structure of the coefficients of $\qbinom{n}{k}$ has been described in terms of the sections $\mathsf{S}_i$ of that $q$-binomial coefficient, and the repeating period in each section has been shown to retain some of the properties of $\mathcal{S}$. 

A good future direction is to determine the minimal quasiperiod of $f_{k}(n)$. It is expected to lie somewhere between $\pi_N(k)$ and $\pi'_N(k)$ but it is unclear how the function actually behaves.

It is also interesting to investigate symmetry in the minimal period of the slopes of the functions $\mathcal{L}_{k}^{(i)}$ --- if we let this period be $P$, we mean to determine when the slope of $\mathcal{L}_{k}^{(i)}$ matches that of $\mathcal{L}_{k}^{(P-i)}$ for $0\le i \le P$. Figure \ref{illustration} gives a counterexample (the slopes for $0\le i < 30$ are not symmetric in this way), but in many examples the pattern holds true.

\subsection*{Acknowledgements}

The author thanks the MIT PRIMES program for making this research possible, as well as Younhun Kim for his guidance and Professor Stanley for suggesting the problem and giving feedback. Additionally, the author would like to thank the anonymous reviewer for several comments that improved the paper. 

This material is based upon work supported by the National Science Foundation under grant no. DMS-1519580.

\end{document}